\documentclass[11pt]{amsart}

\usepackage{amssymb}
\usepackage{lscape}
\usepackage{appendix}
\usepackage[margin=1in]{geometry}
\usepackage[colorlinks]{hyperref}
\usepackage{multirow}
\usepackage{array}

\newtheorem{theorem}{Theorem}[section]
\newtheorem{lem}[theorem]{Lemma}

\theoremstyle{definition}

\theoremstyle{remark}
\newtheorem{remark}[theorem]{Remark}

\numberwithin{equation}{section}

\begin{document}

\newcommand{\spacing}[1]{\renewcommand{\baselinestretch}{#1}\large\normalsize}
\spacing{1.14}

\title[Classification of Douglas $(\alpha,\beta)$-metrics]{Classification of Douglas $(\alpha,\beta)$-metrics on five dimensional nilpotent Lie groups}

\author {M. Hosseini}

\address{ Masoumeh Hosseini \\ Department of Pure Mathematics \\ Faculty of  Mathematics and Statistics\\ University of Isfahan\\ Isfahan\\ 81746-73441-Iran.} \email{hoseini.maasomeh@gmail.com}

\author {H. R. Salimi Moghaddam}

\address{ Hamid Reza Salimi Moghaddam \\ Department of Pure Mathematics \\ Faculty of  Mathematics and Statistics\\ University of Isfahan\\ Isfahan\\ 81746-73441-Iran.} \email{hr.salimi@sci.ui.ac.ir and salimi.moghaddam@gmail.com}

\keywords{five dimensional Lie group, Finsler metric of Douglas type, $(\alpha,\beta)$-metric, sectional and flag curvatures. \\
AMS 2010 Mathematics Subject Classification: 22E60, 53C60.}


\begin{abstract}
In this paper we classify all simply connected five dimensional nilpotent Lie groups which admit $(\alpha,\beta)$-metrics of Berwald and Douglas type defined by a left invariant Riemannian metric and a left invariant vector field. During this classification we give the geodesic vectors, Levi-Civita connection, curvature tensor, sectional curvature and $S$-curvature.
\end{abstract}

\maketitle
\section{\textbf{Introduction}}\label{Intro}
Suppose that $(M,\tilde{a})$ is a connected Riemannian manifold and $I(M)$ denotes its isometry group.
If there exists a nilpotent Lie subgroup $N$ of $I(M)$ such that acts transitively on $M$, $M$ is named
a Riemannian nilmanifold. In \cite{Wilson}, it is shown that for a given homogeneous Riemannian nilmanifold $M$,
there exists a unique nilpotent Lie subgroup of the isometry group, acting simply transitively on the manifold $M$.
Moreover, this subgroup is a normal subgroup of the isometry group. So a homogeneous Riemannian nilmanifold
can be considered as a Riemannian nilpotent Lie group $(N,\tilde{a})$, where $\tilde{a}$ is
a left invariant Riemannian metric.\\
In \cite{Lauret}, Lauret used the Wilson's procedure (see \cite{Wilson}) to classify homogeneous
nilmanifolds of dimensions $3$ and $4$ up to isometry. The same procedure have been used by Homolya and Kowalski
in classification of five dimensional two-step nilmanifolds (see \cite{Homolya-Kowalski}). \\
Recently, the classification of five dimensional nilmanifolds has been completed by Figula and Nagy in \cite{Figula-Nagy}.

In 2002, Deng and Hou generalized the Myers-Steenrod theorem (see \cite{Myers-Steenrod} for Myers-Steenrod theorem)
to Finsler manifolds. In fact they showed that, similar to the Riemannian case, the group of isometries of a
Finsler manifold is a Lie transformation group (see \cite{Deng-Hou}). It is an inflection point in the history of
Finsler geometry because this paper makes the possibility of using Lie theory in Finsler geometry. The study of
homogeneous Finsler manifolds is an interesting field which has grown after the year 2002, for example see
\cite{An-Deng Monatsh, Deng, DeHoLiSa, Deng-Hu-Canadian, Liu-Deng, Nasehi, Salimi} and \cite{Yan-Deng}.
A Finsler metric on a manifold $M$ is a map $F:TM \longrightarrow \left[ 0, \infty \right)$ which satisfies the following conditions
\begin{itemize}
\item[a)] $F$ is smooth on $TM\backslash \{0\}$,
\item[b)] $F$ is a positive homogeneous function of degree one which means that $F(x,\lambda y)=\lambda F(x,y)$, \ \ \ \ for all $\quad  \lambda>0$,
\item[c)] The matrix $(g_{ij})=\left( \dfrac{1}{2} \dfrac{\partial ^2 F^2}{\partial y^i \partial y^j}\right) $, which is called the Hessian matrix, for any $(x,y) \in TM\backslash \{0\}$ is positive definite.
\end{itemize}
In this paper we use left invariant Finsler metrics on nilpotent Lie groups. A left invariant Finsler metric on a
Lie group $G$ is a Finsler metric which for any  $x\in G$ and $y\in T_x G$ holds the following equation
\begin{equation}
F(x,y)=F(e,dl\,_{x^{-1}}y).
\end{equation}
Matsumoto, in his famous article \cite{Matsumoto}, studied an important family of Finsler metrics which is called
$(\alpha, \beta)$-metrics. By definition (see  \cite{Chern-Shen}), an $(\alpha, \beta)$-metric $F$ is a combination of a norm $\alpha (x,y)=\sqrt{\tilde{a}(y,y)}$ defined by a Riemannian metric $\tilde{a}$ and a $1$-form $\beta (x,y)=b_iy^i$, of the form $F = \alpha \phi (\frac{\beta}{\alpha})$, where $\phi : (-b_0 , b_0) \longrightarrow \mathbb{R^+}$ is a $C^\infty$ real function satisfying the following inequality:
\begin{equation}
\phi (s) - s \phi ^{'} (s) + ( b^2 - s^2 ) \phi ^{''} (s) > 0, \qquad  \vert s \vert  \leq b < b_0,
\end{equation}
and for any $x \in M$, $\Vert \beta _x\Vert _{\alpha}=\sqrt{a^{ij} (x)b_i(x) b_j(x) } < b_0$.\\
In the above definition if we put $\phi (s) = 1+s$, then we have the most famous $(\alpha, \beta)$-metric $F=\alpha + \beta$ which is named Randers metric. We mention that for Randers metrics, we have $b_0=1$. \\
In the study of homogeneous $(\alpha, \beta)$-metric spaces, it is very useful to use the dual vector field $X$ of
the $1$-form $\beta$ defined by
\begin{equation}
\tilde{\textsf{a}}(y,X(x)) = \beta (x,y).
\end{equation}
Using this notation, easily we can define left invariant $(\alpha, \beta)$-metrics on Lie groups. For example, let $\tilde{a}$ be a left invariant Riemannian metric on a Lie group $G$. Then for any left invariant vector field $X$ with the condition $\Vert X\Vert_{\alpha}< 1$, the function
\begin{equation}\label{Randers}
F(x,y)=\sqrt{\tilde{a}\left( y ,y \right)}+\tilde{a}\left( X\left( x\right) ,y\right),
\end{equation}
is a left invariant Randers metric on $G$.\\
In a standard local coordinate system $(x,y)=(x^1,\cdots,x^n, y^1,\cdots, y^n)$ of $TM$, the spray coefficients $G^i$ of a Finsler manifold $(M,F)$ are defined by
\begin{equation}\label{spray coefficients}
G^i:=\frac{1}{4}g^{il}([F^2]_{x^my^l}y^m-[F^2]_{x^l}).
\end{equation}
An interesting connection in Finsler geometry is the Chern connection which is defined on the pulled-back bundle $\pi^\ast TM$. In general case the Christoffel symbols $\Gamma^i_{jk}$ of the Chern connection are functions of the points $(x,y)$. If in any standard local coordinate system $(x,y)$, the Christoffel symbols of the Chern connection are only dependent on $x$ then $F$ is named a Berwald metric. In this case the spray coefficients are of the form $G^i=\frac{1}{2}\Gamma^i_{jk}(x)y^jy^k$.\\
A more general family of Finsler metrics is the family of Douglas metrics (see \cite{Chern-Shen}). $F$ is said a Douglas metric if in any standard local coordinate system, there exists a local positively homogeneous function of degree one $P(x ,y)$ such that
\begin{equation}\label{spray coefficients of Douglas metrics}
    G^i=\frac{1}{2}\Gamma^i_{jk}(x)y^jy^k+P(x,y)y^i.
\end{equation}
A non-Riemannian quantity, introduced by Shen \cite{Shen}, is S-curvature. It describes the changes rate of the volume along geodesics in Finsler spaces. For a left invariant Randers metric $F(x,y)=\sqrt{\tilde{a}\left( y ,y \right)}+\tilde{a}\left( X\left( x\right) ,y\right)$, defined by a left invariant Riemannian metric $\tilde{a}$, on a n-dimensional Lie group $G$, it is given by
\begin{equation}
S(y)=\dfrac{n+1}{2}\bigg\{\dfrac{\tilde{a}\left( [X,y],\tilde{a}\left( y,X\right)X-y \right) }{F(y)}-\tilde{a}\left( [X,y],X\right) \bigg\}.
\end{equation}
for mor details see \cite{Deng}.\\
The second author of the present article used the classification of five dimensional two-step nilpotent Lie groups given in \cite{Homolya-Kowalski} to characterize the left invariant Randers metrics of Berwald type on five dimensional two-step nilpotent Lie groups (see \cite{Salimi}). In \cite{Nasehi}, these results are generalized to Randers metrics of Douglas type.

Figula and Nagy in \cite{Figula-Nagy} proved that, up to isometry, any five dimensional Riemannian nilmanifold of nilpotency class greater than two is of the form $\mathfrak{l}_{5,5}$, $\mathfrak{l}_{5,6}$, $\mathfrak{l}_{5,7}$ and $\mathfrak{l}_{5,9}$. Therefor, the results of \cite{Figula-Nagy} together with the results of \cite{Homolya-Kowalski} complete the classification of (non-commutative) five dimensional Riemannian nilmanifolds.

In this work, using the above classification, we characterize all left invariant $(\alpha, \beta)$-metrics of Berwald type and Douglas type on five dimensional nilpotent Lie groups which are defined by a left invariant Riemannian metric and a left invariant vector field. In any case, we give the geodesic vectors, Levi-Civita connection, curvature tensor, sectional curvature and $S$-curvature. About flag curvature of $(\alpha, \beta)$-metrics, we mention that in Theorem 3.2 of \cite{DeHoLiSa}, we have shown that at any point there exists three directions such that the flag curvature admits positive, negative and zero values. So, although they are computable, we do not give the flag curvature formulas here.

\textbf{Notation.} During this paper all Lie groups are non-commutative and simply connected. Also, in this paper, we do not consider the Lie algebras which are direct products of Lie algebras of lower dimension. In any case the set $\{E_1,\cdots,E_5\}$ is an orthonormal basis of the Lie algebras with respect to the inner products induced by the considered left invariant Riemannian metrics. Also, we mention that in definition of Lie algebras we only write the non-zero Lie brackets.


\section{\textbf{Five dimensional two-step nilpotent Lie groups}}
In \cite{Homolya-Kowalski}, Homolya and Kowalski, up to isometry, classified the simply connected five dimensional two-step nilpotent Lie groups (Lie algebras) equipped with left invariant metrics as follows:
\begin{description}\label{5-dim 2-step nilpoten}
  \item[(i) 1-dimensional center] $[E_1,E_2]=\lambda E_5$ and $[E_3,E_4]=\mu E_5$, $\lambda\geq\mu>0$,
  \item[(ii) 2-dimensional center] $[E_1,E_2]=\lambda E_4$ and $[E_1,E_3]=\mu E_5$, $\lambda\geq\mu>0$,
  \item[(iii) 3-dimensional center] $[E_1,E_2]=\lambda E_3$,$\lambda>0$.
\end{description}
In \cite{Liu-Deng}, it is proved that a Douglas homogeneous $(\alpha, \beta)$-metric must be a Berwaldian metric or a Douglas Randers metric. We use this result to classify left invariant Douglas $(\alpha, \beta)$-metrics on five dimensional nilpotent Lie groups.\\
Riemannian geometry of these spaces is investigated in \cite{Salimi} and it is shown that among the above classes, only the class (iii) admits left invariant Randers metrics of Berwald type. In the family (iii), a left invariant Randers metric is of Berwald type if and only if, in the formula (\ref{Randers}), $X=y_4E_4+y_5E_5$ and $0<\sqrt{y_4^2+y_5^2}<1$. By attention to this fact that, similar to the Randers metrics, an $(\alpha, \beta)$-metric is a Berwaldian Finsler metric if and only if the vector field $X$ is parallel to the Levi-Civita connection of the Riemannian metric $\tilde{a}$ (see \cite{Matsumoto2}), the only two-step five dimensional nilpotent Lie group which admits left invariant Berwaldian $(\alpha, \beta)$-metrics belongs to the class (iii) and in this case we have $X=y_4E_4+y_5E_5$.

On the other hand, easily we can see all the above three families admit left invariant Randers metrics of Douglas type (see \cite{Nasehi}). It is sufficient to use the fact that a left invariant Randers metric is of Douglas type if and only if, in the equation (\ref{Randers}), $X$ is a left invariant vector field with the length less than one and orthogonal to the derived subalgebra (\cite{An-Deng Monatsh}).\\
So the classification of $(\alpha, \beta)$-metrics of Douglas type, which are defined by a left invariant Riemannian metric and a left invariant vector field, on two-step nilpotent Lie groups is completed.\\
In the following theorem we give the geodesic vectors in any case.

\begin{theorem} \label{geodesic vector}
Let $G$ be a five dimensional two-step nilpotent Lie group equipped with a Randers metric of Douglas type arisen from a left invariant Riemannian metric $\tilde{a}$ and a left invariant vector field $X$. Then,
\begin{enumerate}
\item[(i)] If $G$ is a 1-dimensional center Lie group then $Y$ is a geodesic vector of $(G,F)$ if and only if $Y\in \textit{span}\{E_2,E_3,E_4\}$ or $Y\in \textit{span}\{E_5\}$.
\item[(ii)] If $G$ is a 2-dimensional center Lie group then $Y$ is a geodesic vector $(G,F)$ if and only if $Y\in \textit{span}\{E_1,E_2,E_3\}$ or $Y\in \textit{span}\{E_2,E_3,E_4,E_5\}$ satisfying $\lambda y_2y_4+\mu y_3y_5=0$.
\item[(iii)] If $G$ is a 3-dimensional center Lie group then $Y$ is a geodesic vector $(G,F)$ if and only if $Y\in \textit{span}\{E_1,E_2,E_4,E_5\}$ or $Y\in \textit{span}\{E_3,E_4,E_5\}$.
\end{enumerate}
\end{theorem}
\begin{proof}
By geodesic lemma in \cite{Kowalski-Vanhecke}, $Y$ is a geodesic vector of $(G,\tilde{a})$ if and only if  $\tilde{a}\left( [Y,Z],Y\right) =0$ for all $Z\in \mathfrak{l}_{5,7}$. Equivalently,  $Y$ is a geodesic vector of $(G,\tilde{a})$ if and only if $\tilde{a}\left( [Y,E_i],Y\right) =0$, $i=1,...,5$. Hence, $Y=y_1E_1+y_2E_2+y_3E_3+y_4E_4+y_5E_5$ is a  geodesic vector of $(G,\tilde{a})$ if and only if the following equations hold:
\begin{equation}
 \left\{
\begin{array}{rl}
 & \textit{in case (i)} \qquad \lambda y_2y_5=0,\quad \mu y_4y_5=0, \quad \lambda y_1^2=0, \quad \mu y_3y_5=0,\\
&  \textit{in case (ii)} \qquad \lambda y_2y_4+\mu y_3y_5=0, \quad \lambda y_1y_4=0, \quad \mu y_1y_5=0,\\
& \textit{in case (iii)} \qquad \lambda y_2y_3=0, \quad \lambda y_1y_3=0.
\end{array} \right.
\end{equation}
  Now, Corollary 2.7 in \cite{Yan-Deng} completes the proof.
\end{proof}
About the $S-$curvature, one can find the formulas of all above cases in \cite{Nasehi}.
\section{\textbf{Five dimensional nilpotent Lie groups of nilpotency class greater than two}}
In this section we classify all $(\alpha, \beta)$-metrics of Douglas type, defined by a left invariant Riemannian metric and a left invariant vector field, on five dimensional nilpotent Lie groups of nilpotency class greater than two. In \cite{Figula-Nagy}, it is shown that, up to isometry, five dimensional nilpotent Lie algebras of nilpotency classes three and four are as follows:
\begin{description}
  \item[Four-step nilpotent Lie algebra $\mathfrak{l}_{5,7}$] $[E_1,E_2]=aE_3+bE_4+cE_5$, $[E_1,E_3]=dE_4+fE_5$, $[E_1,E_4]=gE_5$, where $a,d,g>0$ and either $b>0$ or $b=0$, $f\geq0$,
  \item[Four-step nilpotent Lie algebra $\mathfrak{l}_{5,6}$] $[E_1,E_2]=aE_3+bE_4+cE_5$, $[E_1,E_3]=dE_4+fE_5$, $[E_1,E_4]=gE_5$, $[E_2,E_3]=hE_5$ where $a,d,g,h>0$ and either $b>0$ or $b=0$, $f\geq0$,
  \item[Three-step nilpotent Lie algebra $\mathfrak{l}_{5,5}$] $[E_1,E_2]=aE_4+bE_5$, $[E_1,E_3]=cE_5$, $[E_1,E_4]=dE_5$, $[E_2,E_3]=eE_5$ where $a,d,e>0$, $b,c\geq0$,
  \item[Three-step nilpotent Lie algebra $\mathfrak{l}_{5,9}$] $[E_1,E_2]=kE_3+lE_4+mE_5$, $[E_1,E_3]=pE_4$, $[E_2,E_3]=qE_5$ where $k>0$, $q>p>0$ and $l,m\geq0$ or $k,p>0$, $l\geq0$, $m=0$, $q=p$.
\end{description}
In the following four subsections we give Levi-Civita connection, sectional curvature, Douglas $(\alpha, \beta)$-metrics, geodesic vectors and the $S$-curvatures of each case. Unless stated otherwise, all $(\alpha ,\beta)$-metrics in this section have been defined by a left invariant Riemannian metric $\tilde{a}$ and a left invariant vector field $X$.

\subsection{\textbf{Four-step nilpotent Lie groups with Lie algebra $\mathfrak{l}_{5,7}$}}
By attention to this fact that the set $\{E_1,\cdots,E_5\}$ is an orthonormal basis for the Lie algebra
$\mathfrak{l}_{5,7}$, the Levi-Civita connection is as follows:
 \begin{center}
 \fontsize{6}{0}{\selectfont
\begin{tabular}{| c| c| c| c| c| c| }
\hline $\mathfrak{l}_{5,7}$  & $E_1$  & $E_2$ & $E_3$ & $E_4$ & $E_5$\\
\hline  $\nabla _{E_1}$  & $0$ & $\frac{1}{2}(aE_3+bE_4+cE_5)$ & $\frac{1}{2}(dE_4+fE_5-aE_2)$ & $-\frac{1}{2}(bE_2+dE_3-gE_5)$ & $-\frac{1}{2}(cE_2+fE_3+gE_4)$\\
\hline $\nabla _{E_2}$ & $-\frac{1}{2}(aE_3+bE_4+cE_5)$ & $0$ & $\frac{1}{2}aE_1$ & $\frac{1}{2}bE_1$ & $\frac{1}{2}cE_1$\\
\hline $\nabla _{E_3}$ & $-\frac{1}{2}(aE_2+dE_4+fE_5)$ & $\frac{1}{2}aE_1$ & $0$ & $\frac{1}{2}dE_1$ & $\frac{1}{2}fE_1$\\
\hline $\nabla _{E_4}$ & $-\frac{1}{2}(bE_2+dE_3+gE_5)$ &  $\frac{1}{2}bE_1$ & $\frac{1}{2}dE_1$ & $0$ & $\frac{1}{2}gE_1$\\
\hline $\nabla _{E_5}$ & $-\frac{1}{2}(cE_2+fE_3+gE_4)$ & $\frac{1}{2}cE_1$ & $\frac{1}{2}fE_1$ & $\frac{1}{2}gE_1$ &$0$\\
\hline
\end{tabular}}
\end{center}
Using the Levi-Civita connection we have computed the Riemannian curvature tensor which is given in the appendix. Then, a direct computation shows that for the sectional curvature we have:

 \begin{center}
 \fontsize{8}{0}{\selectfont
\begin{tabular}{| c| c| c| c| c| c| }
\hline $\mathfrak{l}_{5,7}$  & $E_1$  & $E_2$ & $E_3$ & $E_4$ & $E_5$\\
\hline  $K(E_1,.)$  & & $-\frac{3}{4}(a^2+b^2+c^2)$ & $-\frac{1}{4}(3d^2+3f^2-a^2)$ & $\frac{1}{4}(b^2+d^2-3g^2)$ & $\frac{1}{4}(c^2+f^2+g^2)$\\
\hline $K(E_2,.)$ & $*$ &  & $\frac{1}{4}a^2$ & $\frac{1}{4}b^2$ & $\frac{1}{4}c^2$\\
\hline $K(E_3,.)$ & $*$ & $*$ &  & $\frac{1}{4}d^2$ & $\frac{1}{4}f^2$\\
\hline $K(E_4,.)$ & $*$ &  $*$ & $*$ &  & $\frac{1}{4}g^2$\\
\hline $K(E_5,.)$ & $*$ & $*$ & $*$ & $*$ &\\
\hline
\end{tabular}}
\end{center}
Using the above tables we can give the Ricci curvature and scalar curvature as follows:
 \begin{center}
 \fontsize{8}{0}{\selectfont
\begin{tabular}{| c| c| c| c| c| c| }
\hline $\mathfrak{l}_{5,7}$  & $E_1$  & $E_2$ & $E_3$ & $E_4$ & $E_5$\\
\hline  $Ric(.)$  & $-\frac{1}{2}(a^2+b^2+c^2+d^2+f^2+g^2)$ & $-\frac{1}{2}(a^2+b^2+c^2)$ & $\frac{1}{2}(a^2-d^2-f^2)$ & $\frac{1}{2}(b^2+d^2-g^2)$ & $\frac{1}{2}(c^2+f^2+g^2)$\\
\hline
\end{tabular}}
\end{center}
$\textsc{scalar curvature}=-\frac{1}{2}(a^2+b^2+c^2+d^2+f^2+g^2)<0$.\\

In the next lemma we characterize left invariant Douglas $(\alpha ,\beta )$-metrics on Lie groups with Lie algebra $\mathfrak{l}_{5,7}$.

\begin{lem} \label{Douglas metric}
Let $G$ be a five dimensional nilpotent Lie group with Lie algebra $\mathfrak{l}_{5,7}$ equipped with a left invariant $(\alpha ,\beta )$-metric $F$.
 Then, $F$ is of Douglas type if and only if $X\in \textit{span}\{E_1,E_2\}$.
\end{lem}
\begin{proof}
By theorem 3.2 of \cite{An-Deng Monatsh}, $F$  is of Douglas type if and only if X is orthogonal to $[\mathfrak{l}_{5,7} , \mathfrak{l}_{5,7}]$. It is clear that,  $[\mathfrak{l}_{5,7} , \mathfrak{l}_{5,7}]=\textit{span}\{E_3,E_4,E_5\}$ and this completes the proof of the lemma.
\end{proof}

In the following theorem we see that the only Douglas $(\alpha ,\beta )$-metric is Randers metric.

\begin{theorem} \label{non-Riemannian metric of Douglas type}
If $F$ is any left invariant non-Riemannian  $(\alpha ,\beta )$-metric of Douglas type on $G$ 
then $F$ must be a  Randers metric.
\end{theorem}
\begin{proof}
Assume that $F$ such an $(\alpha ,\beta )$-metric. Theorem 1.1 of \cite{Liu-Deng} says that $F$ is a Randers metric or of Berwald type.  Suppose $F$ is not a Randers metric. So, $F$ is of Berwald type and by Proposition 4.1 of \cite{DeHoLiSa}, $ X$ is orthogonal to $[\mathfrak{l}_{5,7} , \mathfrak{l}_{5,7}]$ and we have
\begin{equation}
\tilde{a}\left( [X,Z] , Y \right) + \tilde{a}\left( [X,Y] , Z \right) =0,\qquad \forall \, Y,Z \in \mathfrak{l} _{5,7}.
\end{equation}
According to the previous lemma $X=\lambda _1E_1+\lambda _2E_2$ $(\lambda _1,\lambda _2 \in \mathbb{R})$. If we set $Z=E_5$ and $Y=E_4$ in the above equation, we obtain  $\lambda _1=0$. Again  $Z=E_1$ and $Y=E_3$  implies $\lambda _2=0$. So, $F$ is a Riemannian metric which is a contradiction.
\end{proof}

Here the geodesic vectors are given.

\begin{theorem}
Suppose $G$ is a five dimensional nilpotent Lie group with Lie algebra $\mathfrak{l}_{5,7}$ equipped with a Randers metric
of Douglas type. 
Then, $Y=y_1E_1+y_2E_2+y_3E_3+y_4E_4+y_5E_5$ is a geodesic vector of $(G,F)$ if and only if $Y\in \textit{span}\{E_1,E_2\}$ or $Y$ is orthogonal to $E_1$ and $y_2\left( ay_3+by_4+cy_5 \right)+y_3\left( dy_4+fy_5\right) +gy_4y_5=0$.
\end{theorem}
\begin{proof}
By geodesic lemma in \cite{Kowalski-Vanhecke}, $Y$ is a geodesic vector of $(G,\tilde{a})$ if and only if  $\tilde{a}\left( [Y,z],Y\right) =0$ for all $z\in \mathfrak{l}_{5,7}$. Equivalently,  $Y$ is a geodesic vector of $(G,\tilde{a})$ if and only if $\tilde{a}\left( [Y,E_i],Y\right) =0$, $i=1,...,5$. Hence, $Y=y_1E_1+y_2E_2+y_3E_3+y_4E_4+y_5E_5$ is a  geodesic vector of $(G,\tilde{a})$ if and only if the following equations hold.
\begin{equation}
 \left\{
\begin{array}{rl}
 & y_2\left( ay_3+by_4+cy_5 \right)+y_3\left( dy_4+fy_5\right) +gy_4y_5=0,\\

& y_1\left( ay_3+by_4+cy_5 \right)=0,\\
& y_1\left( dy_4+fy_5\right)=0,\\
&gy_1y_5=0.
\end{array} \right.
\end{equation}
This implies that if $y_1\neq 0$ then $Y\in \textit{span}\{E_1,E_2\}$ and otherwise $y_2\left( ay_3+by_4+cy_5 \right)+y_3\left( dy_4+fy_5\right) +gy_4y_5=0$. 
  Now, Corollary 2.7 in \cite{Yan-Deng} completes the proof.
\end{proof}
Now, we compute the $S$-curvature.
\begin{theorem} \label{S-curvature}
Let $(G,F)$ be a five dimensional nilpotent Lie group with Lie algebra $\mathfrak{l}_{5,7}$ and $F$ be a left invariant non-Riemannian Randers metric of Douglas type, induced by a left invariant Riemannian metric $\tilde{a} $ and a left invariant vector field $X=\lambda _1E_1+\lambda _2E_2$.
Then, the $S$-curvature is given by
\begin{equation} \label{S-curvature5,7}
S(Y)=3\bigg\{ \dfrac{(y_1\lambda _2-y_2\lambda _1)(ay_3+by_4+cy_5)-\lambda _1y_3(dy_4+fy_5)-\lambda _1gy_4y_5}{F(Y)}\bigg\},
\end{equation}
where $Y=y_1E_1+y_2E_2+y_3E_3+y_4E_4+y_5E_5$.
\end{theorem}
\begin{proof}
According to the proposition 7.5 in \cite{Deng} Randers metric $F$ on $G$ has $S$-curvature
\begin{equation}
S(Y)=3\bigg\{\dfrac{\tilde{a}\left( [X,Y],\tilde{a}\left( Y,X\right)X-Y \right) }{F}-\tilde{a}\left( [X,Y],X\right) \bigg\}.
\end{equation}
Since $X$ orthogonal to $[\mathfrak{l}_{5,7},\mathfrak{l}_{5,7}]$, the above formula reduces to
\begin{equation}
S(Y)=3\bigg\{\dfrac{\tilde{a}\left( [Y,X],Y \right) }{F}\bigg\}.
\end{equation}
By a straightforward computation we get \ref{S-curvature5,7}. This completes the proof.
\end{proof}

\subsection{\textbf{Four-step nilpotent Lie groups with Lie algebra $\mathfrak{l}_{5,6}$}}
If $G$ is a Lie group with Lie algebra $\mathfrak{l}_{5,6}$ then, similar to the previous
subsection, for the Levi-civita connection and sectional curvature we have:

\begin{center}
\fontsize{6}{0}{\selectfont
\begin{tabular}{| c| c| c| c| c| c| }
\hline $\mathfrak{l}_{5,6}$  & $E_1$  & $E_2$ & $E_3$ & $E_4$ & $E_5$\\
\hline  $\nabla _{E_1}$  & $0$ & $\frac{1}{2}(aE_3+bE_4+cE_5)$ & $\frac{1}{2}(dE_4+fE_5-aE_2)$ & $-\frac{1}{2}(bE_2+dE_3-gE_5)$ & $-\frac{1}{2}(cE_2+fE_3+gE_4)$\\
\hline $\nabla _{E_2}$ & $-\frac{1}{2}(aE_3+bE_4+cE_5)$ & $0$ & $\frac{1}{2}(aE_1+hE_5)$ & $\frac{1}{2}bE_1$ & $\frac{1}{2}(cE_1-hE_3)$\\
\hline $\nabla _{E_3}$ & $-\frac{1}{2}(aE_2+dE_4+fE_5)$ & $\frac{1}{2}(aE_1-hE_5)$ & $0$ & $\frac{1}{2}dE_1$ & $\frac{1}{2}(fE_1+hE_2)$\\
\hline $\nabla _{E_4}$ & $-\frac{1}{2}(bE_2+dE_3+gE_5)$ &  $\frac{1}{2}bE_1$ & $\frac{1}{2}dE_1$ & $0$ & $\frac{1}{2}gE_1$\\
\hline $\nabla _{E_5}$ & $-\frac{1}{2}(cE_2+fE_3+gE_4)$ & $\frac{1}{2}(cE_1-hE_3)$ & $\frac{1}{2}(fE_1+hE_2)$ & $\frac{1}{2}gE_1$ &$0$\\
\hline
\end{tabular}}
\end{center}

\hspace{3cm}\\

\begin{center}
\fontsize{8}{0}{\selectfont
\begin{tabular}{| c| c| c| c| c| c| }
\hline $\mathfrak{l}_{5,6}$  & $E_1$  & $E_2$ & $E_3$ & $E_4$ & $E_5$\\
\hline  $K(E_1,.)$  & & $-\frac{3}{4}(a^2+b^2+c^2)$ & $-\frac{1}{4}(3d^2+3f^2-a^2)$ & $\frac{1}{4}(b^2+d^2-3g^2)$ & $\frac{1}{4}(c^2+f^2+g^2)$\\
\hline $K(E_2,.)$ & $*$ &  & $\frac{1}{4}(a^2-3h^2)$ & $\frac{1}{4}b^2$ & $\frac{1}{4}(c^2+h^2)$\\
\hline $K(E_3,.)$ & $*$ & $*$ &  & $\frac{1}{4}d^2$ & $\frac{1}{4}(f^2+h^2)$\\
\hline $K(E_4,.)$ & $*$ &  $*$ & $*$ &  & $\frac{1}{4}g^2$\\
\hline $K(E_5,.)$ & $*$ & $*$ & $*$ & $*$ &\\
\hline
\end{tabular}}
\end{center}
Also, the Ricci curvature and scalar curvature are of the following forms.
 \begin{center}
 \fontsize{5}{0}{\selectfont
\begin{tabular}{| c| c| c| c| c| c| }
\hline $\mathfrak{l}_{5,6}$  & $E_1$  & $E_2$ & $E_3$ & $E_4$ & $E_5$\\
\hline  $Ric(.)$  & $-\frac{1}{2}(a^2+b^2+c^2+d^2+f^2+g^2)$ & $-\frac{1}{2}(a^2+b^2+c^2+h^2)$ & $\frac{1}{2}(a^2-d^2-f^2-h^2)$ & $\frac{1}{2}(b^2+d^2-g^2)$ & $\frac{1}{2}(c^2+f^2+g^2+h^2)$\\
\hline
\end{tabular}}
\end{center}
$\textsc{Scalar curvature}=-\frac{1}{2}(a^2+b^2+c^2+d^2+f^2+g^2+h^2)<0$

\begin{lem}
Let $G$ be a five dimensional nilpotent Lie group with Lie algebra $\mathfrak{l}_{5,6}$ equipped with a left invariant $(\alpha ,\beta )$-metric $F$. 
 Then, $F$ is of Douglas type if and only if $X\in \textit{span}\{E_1,E_2\}$.
\end{lem}
\begin{proof}
It is sufficient to notice the fact that $[\mathfrak{l}_{5,6} , \mathfrak{l}_{5,6}]^{\perp}=\textit{span}\{ E_1,E_2\}$.
\end{proof}
Similar to the case $\mathfrak{l}_{5,7}$, we see that all Douglas $(\alpha ,\beta )$-metrics are of Randers type.
\begin{theorem}
Let $F$ be any left invariant non-Riemannian  $(\alpha ,\beta )$-metric of Douglas type on $G$. 
 Then, $F$ must be a  Randers metric.
\end{theorem}
\begin{proof}
The proof of this theorem is similar to that of theorem \ref{non-Riemannian metric of Douglas type}, so we omit it.
\end{proof}
In the case of $\mathfrak{l}_{5,6}$, the geodesic vectors are given in the next theorem.
\begin{theorem}
Suppose $G$ is a five dimensional nilpotent Lie group with Lie algebra $\mathfrak{l}_{5,6}$ equipped with a Randers
metric of Douglas type. 
Then, $Y=y_1E_1+y_2E_2+y_3E_3+y_4E_4+y_5E_5$ is a geodesic vector of $(G,F)$ if and only if $Y\in \textit{span}\{E_1,E_2\}$ or $Y \in \textit{span}\{E_2,E_3,E_4\}$ and
$ay_2y_3+by_2y_4+dy_3y_4=0$ or $Y\in \textit{span}\{E_5\}$.
\end{theorem}
\begin{proof}
 Similar to the proof of theorem \ref{geodesic vector} we can see $Y=y_1E_1+y_2E_2+y_3E_3+y_4E_4+y_5E_5$ is a  geodesic vector of $(G,\tilde{a})$ if and only if
\begin{equation}
 \left\{
\begin{array}{rl}
 & y_2\left( ay_3+by_4+cy_5 \right)+y_3\left( dy_4+fy_5\right) +gy_4y_5=0,\\

& y_1\left( ay_3+by_4+cy_5 \right)-hy_3y_5=0,\\
& y_1\left( dy_4+fy_5\right)+hy_2y_5=0,\\
&gy_1y_5=0.
\end{array} \right.
\end{equation}
Now, Corollary 2.7 in \cite{Yan-Deng} completes the proof.
\end{proof}
For the $S$-curvature quantity we have:
\begin{theorem}
Assume that $(G,F)$ is a five dimensional nilpotent Lie group with Lie algebra $\mathfrak{l}_{5,6}$ and $F$ be a left invariant non-Riemannian Randers metric of Douglas type, induced by a left invariant Riemannian metric $\tilde{a} $ and a left invariant vector field $X=\lambda _1E_1+\lambda _2E_2$.
The $S$-curvature is given by
\begin{equation} \label{S-curvature5,7}
S(Y)=3\bigg\{ \dfrac{(y_1\lambda _2-y_2\lambda _1)(ay_3+by_4+cy_5)-\lambda _1y_3(dy_4+fy_5)-y_5(h\lambda _2y_3+g\lambda _1y_4)}{F(Y)}\bigg\},
\end{equation}
where $Y=y_1E_1+y_2E_2+y_3E_3+y_4E_4+y_5E_5$.
\end{theorem}


\subsection{\textbf{Three-step nilpotent Lie groups with Lie algebra $\mathfrak{l}_{5,5}$}}
During the next two subsections we study five dimensional three-step nilpotent Lie groups.
For the Levi-Civita connection and sectional curvature of left invariant Riemannian metrics on Lie groups with Lie
algebra $\mathfrak{l}_{5,5}$ we have:

\begin{center}
\fontsize{8}{0}{\selectfont
\begin{tabular}{| c| c| c| c| c| c| }
\hline $\mathfrak{l}_{5,5}$  & $E_1$  & $E_2$ & $E_3$ & $E_4$ & $E_5$\\
\hline  $\nabla _{E_1}$  & $0$ & $\frac{1}{2}(aE_4+bE_5)$ & $\frac{1}{2}cE_5$ & $\frac{1}{2}(dE_5-aE_2)$ & $-\frac{1}{2}(bE_2+cE_3+dE_4)$\\
\hline $\nabla _{E_2}$ & $-\frac{1}{2}(aE_4+bE_5)$  & $0$ & $\frac{1}{2}eE_5$ & $\frac{1}{2}aE_1$ & $\frac{1}{2}(bE_1-eE_3)$\\
\hline $\nabla _{E_3}$ & $-\frac{1}{2}cE_5$ & $-\frac{1}{2}eE_5$ & $0$ & $0$ & $\frac{1}{2}(cE_1+eE_2)$\\
\hline $\nabla _{E_4}$ & $-\frac{1}{2}(aE_2+dE_5)$ &  $\frac{1}{2}aE_1$ & $0$ & $0$ & $\frac{1}{2}dE_1$\\
\hline $\nabla _{E_5}$ & $-\frac{1}{2}(bE_2+cE_3+dE_4)$ & $\frac{1}{2}(bE_1-eE_3)$ & $\frac{1}{2}(cE_1+eE_2)$ & $\frac{1}{2}dE_1$ &$0$\\
\hline
\end{tabular}}
\end{center}

\hspace{3cm}\\
\begin{center}
\fontsize{8}{0}{\selectfont
\begin{tabular}{| c| c| c| c| c| c| }
\hline $\mathfrak{l}_{5,5}$  & $E_1$  & $E_2$ & $E_3$ & $E_4$ & $E_5$\\
\hline  $K(E_1,.)$  & & $-\frac{3}{4}(a^2+b^2)$ & $-\frac{3}{4}c^2$ & $\frac{1}{4}(a^2-3d^2)$ & $\frac{1}{4}(b^2+c^2+d^2)$\\
\hline $K(E_2,.)$ & $*$ &  & $-\frac{3}{4}e^2$ & $\frac{1}{4}a^2$ & $\frac{1}{4}(b^2+e^2)$\\
\hline $K(E_3,.)$ & $*$ & $*$ &  & $0$ & $\frac{1}{4}(c^2+e^2)$\\
\hline $K(E_4,.)$ & $*$ &  $*$ & $*$ &  & $\frac{1}{4}d^2$\\
\hline $K(E_5,.)$ & $*$ & $*$ & $*$ & $*$ &\\
\hline
\end{tabular}}
\end{center}
\hspace{3cm}\\
As we have mentioned above the curvature tensor is given in the appendix.
Using the above tables the Ricci and scalar curvatures are of the following forms:
 \begin{center}
 \fontsize{8}{0}{\selectfont
\begin{tabular}{| c| c| c| c| c| c| }
\hline $\mathfrak{l}_{5,5}$  & $E_1$  & $E_2$ & $E_3$ & $E_4$ & $E_5$\\
\hline  $Ric(.)$  & $-\frac{1}{2}(a^2+b^2+c^2+d^2)$ & $-\frac{1}{2}(a^2+b^2+e^2)$ & $-\frac{1}{2}(c^2+e^2)$ & $\frac{1}{2}(a^2-d^2)$ & $\frac{1}{2}(b^2+c^2+d^2+e^2)$\\
\hline
\end{tabular}}
\end{center}

$\textsc{scalar curvature}=-\frac{1}{2}(a^2+b^2+c^2+d^2+e^2)<0$

\begin{lem}
Let $G$ be a five dimensional nilpotent Lie group with Lie algebra $\mathfrak{l}_{5,5}$ equipped with a left invariant $(\alpha ,\beta )$-metric $F$. 
Then, $F$ is of Douglas type if and only if $X\in \textit{span}\{E_1,E_2,E_3\}$.
\end{lem}
\begin{proof}
It is sufficient to notice the fact that $[\mathfrak{l}_{5,5} , \mathfrak{l}_{5,5}]^{\perp}=\textit{span}\{ E_1,E_2,E_3\}$.
\end{proof}
In the following theorem we see that, similar to the above cases, there is not a non-Randers Douglas $(\alpha ,\beta )$-metric.
\begin{theorem}
Suppose $F$ is any left invariant non-Riemannian  $(\alpha ,\beta )$-metric of Douglas type on $G$. 
 Then, $F$ must be a  Randers metric.
\end{theorem}
\begin{proof}
The proof of this theorem is similar to that of theorem \ref{non-Riemannian metric of Douglas type}, so we omit it.
\end{proof}
\begin{theorem}
Assume that $G$ is a five dimensional nilpotent Lie group with Lie algebra $\mathfrak{l}_{5,5}$ equipped with a Randers metric of Douglas type. 
 $Y$ is a geodesic vector of $(G,F)$ if and only if $Y\in \textit{span}\{E_1,E_2,E_3\}$   or $Y \in \textit{span}\{E_3,E_4\}$ or $Y \in \textit{span}\{E_5\}$.
\end{theorem}
\begin{proof}
 Similar to the proof of theorem \ref{geodesic vector} we can see $Y=y_1E_1+y_2E_2+y_3E_3+y_4E_4+y_5E_5$ is a  geodesic vector of $(G,\tilde{a})$ if and only if
\begin{equation}
 \left\{
\begin{array}{rl}
 & y_2\left( ay_4+by_5 \right)+cy_3y_5+dy_4y_5=0,\\
& y_1\left( ay_4+by_5 \right)-ey_3y_5=0,\\
&c y_1y_5+ey_2y_5=0,\\
&dy_1y_5=0.
\end{array} \right.
\end{equation}
Now, Corollary 2.7 of \cite{Yan-Deng} completes the proof.
\end{proof}
\begin{theorem}
Let $(G,F)$ be a five dimensional nilpotent Lie group with Lie algebra $\mathfrak{l}_{5,5}$ and $F$ be a left invariant non-Riemannian Randers metric of Douglas type, induced by a left invariant Riemannian metric $\tilde{a} $ and a left invariant vector field $X=\lambda _1E_1+\lambda _2E_2+\lambda _3E_3$. Then, the $S$-curvature is given by
\begin{equation} \label{S-curvature5,7}
S(Y)=3\bigg\{ \dfrac{(y_1\lambda _2-y_2\lambda _1)(ay_4+by_5)+y_5\lambda _3(cy_1+ey_2)-y_5\lambda _1(cy_3+dy_4)-e\lambda _2y_3y_5}{F(Y)}\bigg\},
\end{equation}
where $Y=y_1E_1+y_2E_2+y_3E_3+y_4E_4+y_5E_5$.
\end{theorem}
\subsection{\textbf{Three-step nilpotent Lie groups with Lie algebra $\mathfrak{l}_{5,9}$}}
For the Levi-Civita connection and sectional curvature of the last case $\mathfrak{l}_{5,9}$ we have:
 \begin{center}
 \fontsize{8}{0}{\selectfont
\begin{tabular}{| c| c| c| c| c| c| }
\hline $\mathfrak{l} _{5,9}$  & $E_1$  & $E_2$ & $E_3$ & $E_4$ & $E_5$\\
\hline  $\nabla _{E_1}$  & $0$ & $\frac{1}{2}(kE_3+lE_4+mE_5)$ & $\frac{1}{2}(pE_4-kE_2)$ & $-\frac{1}{2}(lE_2+pE_3)$ & $-\frac{1}{2}mE_2$\\
\hline $\nabla _{E_2}$ & $-\frac{1}{2}(kE_3+lE_4+mE_5)$  & $0$ & $\frac{1}{2}(kE_1+qE_5)$ & $\frac{1}{2}lE_1$ & $\frac{1}{2}(mE_1-qE_3)$\\
\hline $\nabla _{E_3}$ & $-\frac{1}{2}(pE_4+kE_2)$ & $\frac{1}{2}(kE_1-qE_5)$ & $0$ & $\frac{1}{2}pE_1$ & $\frac{1}{2}qE_2$\\
\hline $\nabla _{E_4}$ & $-\frac{1}{2}(lE_2+pE_3)$ &  $\frac{1}{2}lE_1$ & $\frac{1}{2}pE_1$ & $0$ & $0$\\
\hline $\nabla _{E_5}$ & $-\frac{1}{2}mE_2$ & $\frac{1}{2}(mE_1-qE_3)$ & $\frac{1}{2}qE_2$ & $0$ &$0$\\
\hline
\end{tabular}}
\end{center}
\hspace{3cm}\\
\begin{center}
\fontsize{8}{0}{\selectfont
\begin{tabular}{| c| c| c| c| c| c| }
\hline $\mathfrak{l} _{5,9}$  & $E_1$  & $E_2$ & $E_3$ & $E_4$ & $E_5$\\
\hline  $K(E_1,.)$  & & $-\frac{3}{4}(k^2+l^2+m^2)$ & $\frac{1}{4}(k^2-3p^2)$ & $\frac{1}{4}(l^2+p^2)$ & $\frac{1}{4}m^2$\\
\hline $K(E_2,.)$ & $*$ &  & $\frac{1}{4}(k^2-3q^2)$ & $\frac{1}{4}l^2$ & $\frac{1}{4}(m^2+q^2)$\\
\hline $K(E_3,.)$ & $*$ & $*$ &  & $\frac{1}{4}p^2$ & $\frac{1}{4}q^2$\\
\hline $K(E_4,.)$ & $*$ &  $*$ & $*$ &  & $0$\\
\hline $K(E_5,.)$ & $*$ & $*$ & $*$ & $*$ &\\
\hline
\end{tabular}}
\end{center}
Also, a direct computation shows that for the Ricci curvature we have:
\begin{center}
\fontsize{8}{0}{\selectfont
\begin{tabular}{| c| c| c| c| c| c| }
\hline $\mathfrak{l} _{5,9}$  & $E_1$  & $E_2$ & $E_3$ & $E_4$ & $E_5$\\
\hline  $Ric(.)$  & $-\frac{1}{2}(k^2+l^2+m^2+p^2)$ & $-\frac{1}{2}(k^2+l^2+m^2+q^2)$ & $-\frac{1}{2}(k^2-p^2-q^2)$ & $\frac{1}{2}(l^2+p^2)$ & $\frac{1}{2}(m^2+q^2)$\\
\hline
\end{tabular}}
\end{center}
And the scalar curvature of this case is given by $\textsc{scalar curvature}=-\frac{1}{2}(k^2+l^2+m^2+p^2+q^2)<0$.
\begin{lem}
Let $G$ be a five dimensional nilpotent Lie group with Lie algebra $\mathfrak{l}_{5,9}$ equipped with a left invariant $(\alpha ,\beta )$-metric $F$. 
 Then, $F$ is of Douglas type if and only if $X\in \textit{span}\{E_1,E_2\}$.
\end{lem}
\begin{proof}
It is sufficient to notice the fact that $[\mathfrak{l}_{5,9} , \mathfrak{l}_{5,9}]=\textit{span}\{ E_3,E_4,E_5\}$.
\end{proof}
\begin{theorem}
Suppose $F$ is any left invariant non-Riemannian  $(\alpha ,\beta )$-metric of Douglas type on $G$. 
Then, $F$ must be a  Randers metric.
\end{theorem}
\begin{proof}
The proof of this theorem is similar to that of theorem \ref{non-Riemannian metric of Douglas type}, so we omit it.
\end{proof}
In the following theorem the geodesic vectors of Douglas Randers metrics are given.
\begin{theorem}
Assume that $G$ is a five dimensional nilpotent Lie group with Lie algebra $\mathfrak{l} _{5,9}$ equipped with a Randers metric of Douglas type. 
$Y$ is a geodesic vector of $(G,F)$ if and only if $Y\in \textit{span}\{E_3\}$   or $Y \in \textit{span}\{E_4,E_5\}$ or $Y \in \textit{span}\{E_2,E_3,E_4\}$ satisfying  $y_2(ky_3+ly_4)+py_3y_4=0 $ or $Y$ is orthogonal to $E_1$ and satisfies the condition
 \begin{equation}
 \left\{
\begin{array}{rl}
& y_1\left( ky_3+ly_4+my_5 \right)-qy_3y_5=0,\\
&p y_1y_4+qy_2y_5=0,
\end{array} \right.
\end{equation}
where $Y=y_1E_1+y_2E_2+y_3E_3+y_4E_4+y_5E_5$.
\end{theorem}
\begin{proof}
 Similar to the proof of theorem \ref{geodesic vector} we can see $Y=y_1E_1+y_2E_2+y_3E_3+y_4E_4+y_5E_5$ is a  geodesic vector of $(G,\tilde{a})$ if and only if
\begin{equation}
 \left\{
\begin{array}{rl}
 & y_2\left( ky_3+ly_4+my_5 \right)+py_3y_4=0,\\
& y_1\left( ky_3+ly_4+my_5 \right)-qy_3y_5=0,\\
&p y_1y_4+qy_2y_5=0.
\end{array} \right.
\end{equation}
 It suffices to investigate two cases  $y_1=0$ and $y_1\neq 0$. Now, Corollary 2.7 in \cite{Yan-Deng} completes the proof.
\end{proof}
And for the $S$-curvature of this space we have the following result:
\begin{theorem}
Let $(G,F)$ be a five dimensional nilpotent Lie group with Lie algebra $\mathfrak{l}_{5,9}$ and $F$ be a left invariant non-Riemannian Randers metric of Douglas type, induced by a left invariant Riemannian metric $\tilde{a} $ and a left invariant vector field $X=\lambda _1E_1+\lambda _2E_2$. Then, the S-curvature is given by
\begin{equation} \label{S-curvature5,7}
S(Y)=3\bigg\{ \dfrac{(y_1\lambda _2-y_2\lambda _1)(ky_3+ly_4+my_5)-y_3(p\lambda _1y_4+q\lambda _2y_5)}{F(Y)}\bigg\},
\end{equation}
where $Y=y_1E_1+y_2E_2+y_3E_3+y_4E_4+y_5E_5$.
\end{theorem}
\begin{remark}
The above theorems show that, against with the five dimensional two-step nilpotent Lie groups, there is not non-Randers Douglas $(\alpha ,\beta )$-metrics on five dimensional nilpotent Lie groups with nilpotency class greater than two.
\end{remark}


\appendix

\begin{center}
\fontsize{8}{0}{\selectfont
\begin{tabular}{| c| c|c|   }
\hline $\mathfrak{l}_{5,7}$  & $E_1$  & $E_2$ \\
\hline  $R(E_1,E_2)$  & $\frac{1}{4}(3(a^2+b^2+c^2)E_2+3(bd+cf)E_3+(3cg+ad)E_4+(af+bg)E_5)$ & $-\frac{3}{4}(a^2+b^2+c^2)E_1$ \\
\hline $R(E_1,E_3)$ & $\frac{1}{4}(3(bd+cf)E_2+(3f^2+3d^2-a^2)E_3+(3fg-ab)E_4+(dg-ac)E_5)$  & $-\frac{3}{4}(bd+cf)E_1$ \\
\hline $R(E_1,E_4)$ & $\frac{1}{4}((ad+3cg)E_2+(3fg-ab)E_3+(3g^2-d^2-b^2)E_4-(bc+fd)E_5)$ & $-\frac{1}{4}(ad+3cg)E_1$ \\
\hline $R(E_1,E_5)$ & $\frac{1}{4}((af+bg)E_2+(dg-ac)E_3-(bc+df)E_4-(c^2+f^2+g^2)E_5)$ &  $-\frac{1}{4}(af+bg)E_1$\\
\hline $R(E_2,E_3)$ & $0$ &  $-\frac{1}{4}(a^2E_3+abE_4+acE_5)$\\
\hline $R(E_2,E_4)$ & $0$ &  $-\frac{1}{4}(abE_3+b^2E_4+bcE_5)$ \\
\hline $R(E_2,E_5)$ & $0$ &  $-\frac{1}{4}(acE_3+bcE_4+c^2E_5)$ \\
\hline $R(E_3,E_4)$ & $0$ &  $\frac{1}{4}(adE_3-bdE_4+(ag-bf)E_5)$ \\
\hline $R(E_3,E_5)$ & $0$ &  $\frac{1}{4}(afE_3+(ag-cd)E_4-cfE_5)$ \\
\hline$R(E_4,E_5)$ & $0$ &  $\frac{1}{4}((bf-cd)E_3+bgE_4-cgE_5)$ \\
\hline
\end{tabular}
}
\end{center}
\hspace{2cm}\\
\begin{center}
\fontsize{8}{0}{\selectfont
\begin{tabular}{| c| c|c| c|  }
\hline $\mathfrak{l}_{5,7}$  & $E_3$  & $E_4$&$E_5$ \\
\hline  $R(E_1,E_2)$  & $-\frac{3}{4}(bd+cf)E_1$ & $-\frac{1}{4}(3ad+cg)E_1$ & $-\frac{1}{4}(af+bg)E_1$\\
\hline $R(E_1,E_3)$ & $-\frac{1}{4}(3d^2+3f^2-a^2)E_1$ & $-\frac{1}{4}(3fg-ab)E_1$ & $\frac{1}{4}(ac-dg)E_1$\\
\hline $R(E_1,E_4)$& $-\frac{1}{4}(3fg-ab)E_1$ & $\frac{1}{4}(b^2+d^2-3g^2)E_1$ & $\frac{1}{4}(bc+df)E_1$\\
\hline $R(E_1,E_5)$  & $\frac{1}{4}(ac-dg)E_1$ & $\frac{1}{4}(bc+df)E_1$ & $\frac{1}{4}(c^2+f^2+g^2)E_1$\\
\hline $R(E_2,E_3)$ & $\frac{1}{4}(a^2E_2+adE_4+afE_5)$ & $\frac{1}{4}(abE_2-adE_3+(bf-cd)E_5)$ & $\frac{1}{4}(acE_2-afE_3+(cd-bf)E_4)$\\
\hline $R(E_2,E_4)$ & $\frac{1}{4}(abE_2-bdE_4+(ag-cd)E_5)$ & $\frac{1}{4}(b^2E_2+bdE_3+bgE_5)$ & $\frac{1}{4}(bcE_2+(cd-ag)E_3-bgE_4)$\\
\hline $R(E_2,E_5)$ & $\frac{1}{4}(acE_2+(ag-bf)E_4-cfE_5)$ & $\frac{1}{4}(bcE_2+(bf-ag)E_3-cgE_5)$ & $\frac{1}{4}(c^2E_2+cfE_3+cgE_4)$\\
\hline $R(E_3,E_4)$ & $-\frac{1}{4}(adE_2+d^2E_4+dfE_5)$ & $\frac{1}{4}(bdE_2+d^2E_3+dgE_5)$ & $\frac{1}{4}((bf-ag)E_2+dfE_3-dgE_4)$\\
\hline $R(E_3,E_5)$ & $-\frac{1}{4}(afE_2+dfE_4+f^2E_5)$ & $\frac{1}{4}((cd-ag)E_2+dfE_3-fgE_5)$ & $\frac{1}{4}(cfE_2+f^2E_3+fgE_4)$\\
\hline$R(E_4,E_5)$ & $\frac{1}{4}((cd-bf)E_2+dgE_4-fgE_5)$ & $-\frac{1}{4}(bgE_2+dgE_3+g^2E_5)$ & $\frac{1}{4}(cgE_2+fgE_3+g^2E_4)$\\
\hline
\end{tabular}
}
\end{center}
\hspace{6cm}\\
\begin{center}
\fontsize{6}{0}{\selectfont
\begin{tabular}{| c| c|c|   }
\hline $\mathfrak{l}_{5,6}$  & $E_1$  & $E_2$ \\
\hline  $R(E_1,E_2)$  & $\frac{1}{4}(3(a^2+b^2+c^2)E_2+3(bd+cf)E_3+(3cg+ad)E_4+(af+bg)E_5)$ & $-\frac{1}{4}(3(a^2+b^2+c^2)E_1-3chE_3-ahE_5)$ \\
\hline $R(E_1,E_3)$ & $\frac{1}{4}(3(bd+cf)E_2+(3f^2+3d^2-a^2)E_3+(3fg-ab)E_4+(dg-ac)E_5)$  & $-\frac{1}{4}(3(bd+cf)E_1-3fhE_3-ghE_4)$ \\
\hline $R(E_1,E_4)$ & $\frac{1}{4}((ad+3cg)E_2+(3fg-ab)E_3+(3g^2-d^2-b^2)E_4-(bc+fd)E_5)$ & $-\frac{1}{4}(ad-3cg)E_1+\frac{1}{2}ghE_3$ \\
\hline $R(E_1,E_5)$ & $\frac{1}{4}((af+bg)E_2+(dg-ac)E_3-(bc+df)E_4-(c^2+f^2+g^2)E_5)$ &  $-\frac{1}{4}((af+bg)E_1+dhE_4+fhE_5)$\\
\hline $R(E_2,E_3)$ & $\frac{1}{4}(3chE_2+3fhE_3+2hgE_4)$ &  $-\frac{1}{4}(3chE_1+(a^2-3h^2)E_3+abE_4+acE_5)$\\
\hline $R(E_2,E_4)$ & $\frac{1}{4}(ghE_3-dhE_5)$ &  $-\frac{1}{4}(abE_3+b^2E_4+bcE_5)$ \\
\hline $R(E_2,E_5)$ & $\frac{1}{4}(ahE_2-fhE_5)$ &  $-\frac{1}{4}(ahE_1+acE_3+bcE_4+(c^2+h^2)E_5)$ \\
\hline $R(E_3,E_4)$ & $-\frac{1}{4}(ghE_2-bhE_5)$ &  $\frac{1}{4}(ghE_1+adE_3-bdE_4+(ag-bf)E_5)$ \\
\hline $R(E_3,E_5)$ & $-\frac{1}{4}(ahE_3-chE_5)$ &  $\frac{1}{4}(afE_3+(ag-cd)E_4-cfE_5)$ \\
\hline$R(E_4,E_5)$ & $\frac{1}{4}(dhE_2-bhE_3)$ &  $-\frac{1}{4}(dhE_1+(cd-bf)E_3-bgE_4+cgE_5)$ \\
\hline
\end{tabular}}
\end{center}
\hspace{2cm}\\
\begin{center}
\fontsize{2}{0}{\selectfont
\begin{tabular}{| c| c|c| c|  }
\hline $\mathfrak{l}_{5,6}$  & $E_3$  & $E_4$&$E_5$ \\
\hline  $R(E_1,E_2)$  & $-\frac{1}{4}(3(bd+cf)E_1+3chE_2+ghE_4)$ & $-\frac{1}{4}(3(ad+cg)E_1-ghE_3+dhE_5)$ & $-\frac{1}{4}
((af+bg)E_1+ahE_2+dhE_4)$\\
\hline $R(E_1,E_3)$ & $-\frac{1}{4}((3d^2+3f^2-a^2)E_1+3fhE_2+ahE_5)$ & $-\frac{1}{4}((3fg-ab)E_1+ghE_2+bhE_5)$ & $\frac{1}{4}((ac-dg)E_1+ahE_3+bhE_4)$\\
\hline $R(E_1,E_4)$& $-\frac{1}{4}((3fg-ab)E_1+2ghE_2)$ & $\frac{1}{4}(b^2+d^2-3g^2)E_1$ & $\frac{1}{4}(bc+df)E_1$\\
\hline $R(E_1,E_5)$  & $\frac{1}{4}((ac-dg)E_1+bhE_4+chE_5)$ & $\frac{1}{4}((bc+df)E_1+dhE_2-bhE_3)$ & $\frac{1}{4}((c^2+f^2+g^2)E_1+fhE_2-chE_3)$\\
\hline $R(E_2,E_3)$ & $-\frac{1}{4}(3fhE_1+(3h^2-a^2)E_2-adE_4-afE_5)$ & $-\frac{1}{4}(2ghE_1-abE_2+adE_3+(cd-bf)E_5)$ & $\frac{1}{4}(acE_2-afE_3+(cd-bf)E_4)$\\
\hline $R(E_2,E_4)$ & $-\frac{1}{4}(ghE_1-abE_2+bdE_4+(cd-ag)E_5)$ & $\frac{1}{4}(b^2E_2+bdE_3+bgE_5)$ & $\frac{1}{4}(dhE_1+bcE_2+(cd-ag)E_3-bgE_4)$\\
\hline $R(E_2,E_5)$ & $\frac{1}{4}(acE_2+(ag-bf)E_4-cfE_5)$ & $\frac{1}{4}(bcE_2+(bf-ag)E_3-cgE_5)$ & $\frac{1}{4}(fhE_1+(c^2+h^2)E_2+cfE_3+cgE_4)$\\
\hline $R(E_3,E_4)$ & $-\frac{1}{4}(adE_2+d^2E_4+dfE_5)$ & $\frac{1}{4}(bdE_2+d^2E_3+dgE_5)$ & $-\frac{1}{4}(bhE_1+(ag-bf)E_2-dfE_3+dgE_4)$\\
\hline $R(E_3,E_5)$ & $\frac{1}{4}(ahE_1-afE_2-dfE_4-(f^2+h^2)E_5)$ & $\frac{1}{4}((cd-ag)E_2+dfE_3-fgE_5)$ & $-\frac{1}{4}(chE_1-cfE_2-(f^2+h^2)E_3-fgE_4)$\\
\hline$R(E_4,E_5)$ & $\frac{1}{4}(bhE_1+(cd-bf)E_2+dgE_4-fgE_5)$ & $-\frac{1}{4}(bgE_2+dgE_3+g^2E_5)$ & $\frac{1}{4}(cgE_2+fgE_3+g^2E_4)$\\
\hline
\end{tabular}}
\end{center}
\hspace{6cm}\\
\begin{center}
\fontsize{8}{0}{\selectfont
\begin{tabular}{| c| c|c|   }
\hline $\mathfrak{l}_{5,5}$  & $E_1$  & $E_2$ \\
\hline  $R(E_1,E_2)$  & $\frac{1}{4}(3(a^2+b^2)E_2+3bcE_3+3bdE_4+adE_5)$ & $-\frac{3}{4}((a^2+b^2)E_1-beE_3)$ \\
\hline $R(E_1,E_3)$ & $\frac{3}{4}(bcE_2+c^2E_3+cdE_4)$  & $-\frac{1}{4}(3bcE_1-3ceE_3-deE_4)$ \\
\hline $R(E_1,E_4)$ & $\frac{1}{4}(3bdE_2+3cdE_3+(3d^2-a^2)E_4-abE_5)$ & $-\frac{1}{4}(3bdE_1-2edE_3)$ \\
\hline $R(E_1,E_5)$ & $\frac{1}{4}(adE_2-abE_4-(b^2+c^2+d^2)E_5)$ &  $-\frac{1}{4}(adE_1+ceE_5)$\\
\hline $R(E_2,E_3)$ & $\frac{1}{4}(3beE_2+3ceE_3+2deE_4)$ &  $-\frac{3}{4}(beE_1-e^2E_3)$\\
\hline $R(E_2,E_4)$ & $\frac{1}{4}deE_3$ &  $-\frac{1}{4}(a^2E_4+abE_5)$ \\
\hline $R(E_2,E_5)$ & $-\frac{1}{4}ceE_5$ &  $-\frac{1}{4}(abE_4+(b^2+e^2)E_5)$ \\
\hline $R(E_3,E_4)$ & $-\frac{1}{4}(deE_2-aeE_5)$ &  $\frac{1}{4}(deE_1-acE_5)$ \\
\hline $R(E_3,E_5)$ & $\frac{1}{4}beE_5$ &  $-\frac{1}{4}bcE_5$ \\
\hline$R(E_4,E_5)$ & $-\frac{1}{4}aeE_3$ &  $\frac{1}{4}(acE_3+adE_4-bdE_5)$ \\
\hline
\end{tabular}}
\end{center}
\hspace{2cm}\\
\begin{center}
\fontsize{8}{0}{\selectfont
\begin{tabular}{| c| c|c| c|  }
\hline $\mathfrak{l}_{5,5}$  & $E_3$  & $E_4$&$E_5$ \\
\hline  $R(E_1,E_2)$  & $-\frac{1}{4}(3bcE_1+3beE_2+deE_4)$ & $-\frac{1}{4}(3bdE_1-deE_3)$ & $0$\\
\hline $R(E_1,E_3)$ & $-\frac{3}{4}(c^2E_1+ceE_2)$ & $-\frac{1}{4}(3cdE_1+deE_2+aeE_5)$ & $\frac{1}{4}aeE_4$\\
\hline $R(E_1,E_4)$& $-\frac{1}{4}(3cdE_1+2deE_2)$ & $\frac{1}{4}(a^2-3d^2)E_1$ & $\frac{1}{4}abE_1$\\
\hline $R(E_1,E_5)$  & $\frac{1}{4}(aeE_4+beE_5)$ & $\frac{1}{4}(abE_1-aeE_3)$ & $\frac{1}{4}((b^2+c^2+d^2)E_1+ceE_2-beE_3)$\\
\hline $R(E_2,E_3)$ & $-\frac{3}{4}(ceE_1+e^2E_2)$ & $-\frac{1}{4}(2deE_1-acE_5)$ & $-\frac{1}{4}acE_4$\\
\hline $R(E_2,E_4)$ & $-\frac{1}{4}deE_1$ & $\frac{1}{4}(a^2E_2+adE_5)$ & $\frac{1}{4}(abE_2-adE_4)$\\
\hline $R(E_2,E_5)$ & $-\frac{1}{4}(acE_4+bcE_5)$ & $\frac{1}{4}(abE_2+acE_3-bdE_5)$ & $\frac{1}{4}(ceE_1+(b^2+e^2)E_2+bcE_3+bdE_4)$\\
\hline $R(E_3,E_4)$ & $0$ & $0$ & $-\frac{1}{4}(aeE_1-acE_2)$\\
\hline $R(E_3,E_5)$ & $-\frac{1}{4}(c^2+e^2)E_5$ & $-\frac{1}{4}cdE_5$ & $-\frac{1}{4}(beE_1-bcE_2-(c^2+e^2)E_3-cdE_4)$\\
\hline$R(E_4,E_5)$ & $\frac{1}{4}(aeE_1-acE_2-cdE_5)$ & $-\frac{1}{4}(adE_2+d^2E_5)$ & $\frac{1}{4}(bdE_2+cdE_3+d^2E_4)$\\
\hline
\end{tabular}}
\end{center}
\hspace{6cm}\\
\begin{center}
\fontsize{8}{0}{\selectfont
\begin{tabular}{| c| c|c|   }
\hline $\mathfrak{l} _{5,9}$  & $E_1$  & $E_2$ \\
\hline  $R(E_1,E_2)$  & $\frac{1}{4}(3(k^2+l^2+m^2)E_2+3lpE_3+kpE_4)$ & $-\frac{1}{4}(3(k^2+l^2+m^2)E_1-3mqE_3-kqE_5)$ \\
\hline $R(E_1,E_3)$ & $\frac{1}{4}(3lpE_2+(3p^2-k^2)E_3-klE_4-kmE_5)$  & $-\frac{3}{4}lpE_1$ \\
\hline $R(E_1,E_4)$ & $\frac{1}{4}(kpE_2-klE_3-(l^2+p^2)E_4-lmE_5)$ & $-\frac{1}{4}kpE_1$ \\
\hline $R(E_1,E_5)$ & $-\frac{1}{4}(kmE_3+lmE_4+m^2E_5)$ &  $-\frac{1}{4}pqE_4$\\
\hline $R(E_2,E_3)$ & $\frac{3}{4}mqE_2$ &  $-\frac{1}{4}(3mqE_1+(k^2-3q^2)E_3+klE_4+mkE_5)$\\
\hline $R(E_2,E_4)$ & $-\frac{1}{4}pqE_5$ &  $-\frac{1}{4}(klE_3+l^2E_4+mlE_5)$ \\
\hline $R(E_2,E_5)$ & $\frac{1}{4}kqE_2$ &  $-\frac{1}{4}(kqE_1+kmE_3+lmE_4+(m^2+q^2)E_5)$ \\
\hline $R(E_3,E_4)$ & $\frac{1}{4}lqE_5$ &  $\frac{1}{4}(kpE_3-lpE_4)$ \\
\hline $R(E_3,E_5)$ & $-\frac{1}{4}(kqE_3-mqE_5)$ &  $-\frac{1}{4}mpE_4$ \\
\hline$R(E_4,E_5)$ & $\frac{1}{4}(pqE_2-lqE_3)$ &  $-\frac{1}{4}(pqE_1+mpE_3)$ \\
\hline
\end{tabular}}
\end{center}
\hspace{2cm}\\
\begin{center}
\fontsize{8}{0}{\selectfont
\begin{tabular}{| c| c|c| c|  }
\hline $\mathfrak{l} _{5,9}$  & $E_3$  & $E_4$&$E_5$ \\
\hline  $R(E_1,E_2)$  & $-\frac{3}{4}(lpE_1+mqE_2)$ & $-\frac{1}{4}(kpE_1-pqE_5)$ & $-\frac{1}{4}
(kqE_2+pqE_4)$\\
\hline $R(E_1,E_3)$ & $\frac{1}{4}((k^2-3p^2)E_1-kqE_5)$ & $\frac{1}{4}(klE_1-lqE_5)$ & $\frac{1}{4}(kmE_1+kqE_3+lqE_4)$\\
\hline $R(E_1,E_4)$& $\frac{1}{4}klE_1$ & $\frac{1}{4}(l^2+p^2)E_1$ & $\frac{1}{4}lmE_1$\\
\hline $R(E_1,E_5)$  & $\frac{1}{4}(kmE_1+lqE_4+mqE_5)$ & $\frac{1}{4}(lmE_1+pqE_2-lqE_3)$ & $\frac{1}{4}(m^2E_1-mqE_3)$\\
\hline $R(E_2,E_3)$ & $\frac{1}{4}((k^2-3q^2)E_2+kpE_4)$ & $\frac{1}{4}(klE_2-kpE_3-mpE_5)$ & $\frac{1}{4}(kmE_2+mpE_4)$\\
\hline $R(E_2,E_4)$ & $\frac{1}{4}(klE_2-lpE_4-mpE_5)$ & $\frac{1}{4}(l^2E_2+lpE_3)$ & $\frac{1}{4}(pqE_1+lmE_2+mpE_3)$\\
\hline $R(E_2,E_5)$ & $\frac{1}{4}kmE_2$ & $\frac{1}{4}lmE_2$ & $\frac{1}{4}(m^2+q^2)E_2$\\
\hline $R(E_3,E_4)$ & $-\frac{1}{4}(kpE_2+p^2E_4)$ & $\frac{1}{4}(lpE_2+p^2E_3)$ & $-\frac{1}{4}lqE_1$\\
\hline $R(E_3,E_5)$ & $\frac{1}{4}(kqE_1-q^2E_5)$ & $\frac{1}{4}mpE_2$ & $-\frac{1}{4}(mqE_1-q^2E_3)$\\
\hline$R(E_4,E_5)$ & $\frac{1}{4}(lqE_1+mpE_2)$ & $0$ & $0$\\
\hline
\end{tabular}}
\end{center}
\end{document}